\documentclass{article}

\usepackage[utf8]{inputenc}
\usepackage[T1]{fontenc}
\usepackage{lmodern}
\usepackage{amsmath}
\usepackage{amssymb}
\usepackage{amsthm}
\usepackage{graphicx}
\usepackage{geometry}
\usepackage{fancyhdr}
\usepackage{dsfont}
\usepackage{verbatim}
\usepackage{hyperref}
\usepackage{nicefrac}
\usepackage{sectsty}
\usepackage{abstract}

\usepackage{enumerate}

\usepackage[capitalise]{cleveref} 

\hypersetup{
    colorlinks,
    linkcolor={red},
    citecolor={green},
    urlcolor={blue}
}

\allsectionsfont{\bfseries\sffamily}

\newtheoremstyle{cap}
{}
{}
{\itshape}
{}
{\scshape}
{.}
{4pt}
{}
\theoremstyle{cap}
\newtheorem{theorem}{Theorem}[section]

\newtheorem{lemma}[theorem]{Lemma}
\newtheorem{corollary}[theorem]{Corollary}
\newtheorem{prop}[theorem]{Proposition}
\newtheorem{assumption}{Assumption}

\theoremstyle{definition}

\theoremstyle{remark}

\newtheorem*{remark}{Remark}

\providecommand{\keywords}[1]
{
  \textbf{\textit{Keywords---}} #1 \\
}

\providecommand{\msc}[1]
{
  \textbf{\textit{Mathematics Subject Classification---}} #1 \\
}

\title{Convergence Rates for Empirical Measures of Markov Chains in Dual and Wasserstein Distances}
\author{Adrian Riekert\footnote{Applied Mathematics Münster: Institute for Analysis and Numerics, University of Münster, Germany; e-mail: \texttt{ariekert}\textcircled{\texttt{a}}\texttt{uni-muenster.de}}
}

\begin{document}

\setlength{\parindent}{0pt}

\renewcommand{\labelenumi}{\roman{enumi})}

\newcommand{\N}{\mathbb{N}}
\newcommand{\R}{\mathbb{R}}

\newcommand{\Z}{\mathbb{Z}}
\newcommand{\prob}{\mathbb{P}}
\newcommand{\E}{\mathbb{E}}

\newcommand{\dia}{\operatorname{diam}}
\newcommand{\var}{\operatorname{Var}}

\newcommand{\dx}{\mathrm{d}}
\newcommand{\X}{\mathcal{X}}
\newcommand{\cP}{\mathcal{P}}

\maketitle
\begin{abstract}
We consider a Markov chain on $\R^d$ with invariant measure $\mu$. We are interested in the rate of convergence of the empirical measures towards the invariant measure with respect to various dual distances,
including in particular the $1$-Wasserstein distance. The main result of this article is a new upper bound for the expected distance, which is proved by combining a Fourier expansion with a truncation argument. Our bound matches the known rates for i.i.d.~random variables up to logarithmic factors.
In addition, we show how concentration inequalities around the mean can be obtained.
\end{abstract}
\msc{60B10, 60J05, 65C05}
\keywords{Empirical Measure, Markov Chains, Wasserstein Distance, Concentration}
\section{Introduction and main results}
\subsection{Empirical measures}
Let $X_0, X_1, X_2, \ldots$ be a Markov chain on $\R^d$ with invariant probability distribution $\mu$. For $n \in \N$ we define the empirical measure
\begin{equation*}
\mu_n = \frac{1}{n} \sum_{i=1}^n \delta_{X_i},
\end{equation*}
a random probability measure on $\R^d $.
Under suitable conditions these measures will converge to $\mu$ as $n \to \infty$. The purpose of this article is to quantify the rate of convergence with respect to various dual distances of the form
\begin{equation}
d_S ( \mu , \nu ) = \sup\nolimits_{f \in S} | \mu ( f ) - \nu ( f ) |
\end{equation}
for a suitable class of functions $S \subseteq C ( \R^d , \R )$.
In particular, if $S $ is the class of $1$-Lipschitz functions then by Kantorovich duality $d_S ( \mu , \nu )$ agrees with $W_1 ( \mu , \nu ) $, the $1$-Wasserstein distance (see Villani~\cite{Villani.2009}).
Analyzing the convergence rate of empirical measures with respect to the Wasserstein distance is a classical problem with numerous applications, including, e.g., clustering (\cite{Laloe.2010}), density estimation (\cite{Bolley.2006}), and Monte Carlo integration. 
The special case where the $X_i$ are i.i.d.\ with common distribution $\mu$ has been studied extensively, see, e.g., \cite{Dudley.1969,Ajtai.1984,Dereich.2013,Boissard.2014,Weed.2019}. The most general and tight results (\cite{Fournier.2015, Lei.2020}) state that $\E [W_1(\mu, \mu_n)]$ is of order $n^{-\nicefrac{1}{d}}$ if $d \geq 3$ and $\mu$ has a finite $q$-th moment for some $q > \frac{d}{d-1}$.  \\
Some of the proofs for the i.i.d.\ case can be adapted to Markov chains, but usually only under strong additional assumptions, such as absolute continuity of the initial distribution with respect to $\mu$ (see, e.g., Fournier \& Guillin~\cite{Fournier.2015} and Boissard \& Le Gouic~\cite{Boissard.2014}). 
This requires that one already has access to some approximation of $\mu$ to start the Markov chain with, which is not always the case in applications.
We will instead follow the approach from Kloeckner~\cite{Kloeckner.2018}, which does not need such an assumption on the initial distribution.
Since the convergence rate $n^{- \nicefrac{1}{d} }$ becomes extremely slow for large dimensions $d$ -- a phenomenon often referred to as the "curse of dimensionality" -- we will also consider different function classes with stronger smoothness assumptions, which significantly improves the convergence rate in high dimensions.

\subsection{Contractive Markov chains}

Let $\cP(\R^d)$ denote the set of Borel probability measures on $\R^d$ and $\cP_1 ( \R^d)$ the set of measures in  $\cP(\R^d)$ with a finite first moment. 
For $s \in \N$ let $C^s ( \R^d)$ be the set of $(s-1)$-times continuously differentiable functions with a Lipschitz continuous $(s-1)$-th derivative. On $C^s$ we consider the norm
\begin{equation*}
\| f \| _{C^s } = \max_{k \in \{ 1, \ldots, s-1\}} \max_{\alpha \in \{1, \ldots, d \}^k} \left\| \frac{\partial^k f }{\partial x_{\alpha_1} \ldots \partial x_{\alpha_k } } \right\| _\infty 
+ \max_{\alpha \in \{1, \ldots, d \}^{s-1}} \operatorname{Lip} \left( \frac{\partial^{s - 1 } f }{\partial x_{\alpha_1} \ldots \partial x_{\alpha_{s-1} } } \right),
\end{equation*}
where $\operatorname{Lip}$ denotes the Lipschitz constant.
Let $C^s_1 ( \R^d ) := \{ f \in C^s ( \R^d ) \colon \| f \|_{C^s } \} \le 1$.
In particular, $C^1 _ 1 ( \R ^d )$ is the set of $1$-Lipschitz functions.
Since functions in $C^s_1 ( \R^d )$ are integrable with respect to any measure $\mu \in \cP_1 ( \R^d )$, 
the distance 
\begin{equation}
d_{C^s_1 ( \R^d )} ( \mu , \nu ) = \sup\nolimits_{f \in C^s _ 1 } | \mu ( f ) - \nu ( f ) |
\end{equation}
is well-defined for $\mu , \nu \in \cP_1 ( \R^d )$. For $s=1$ this is precisely the $1$-Wasserstein distance.

Let $P(x, \dx y)$ be a Markov kernel on $\R^d$ and $X_0, X_1, X_2, \ldots$ the corresponding time-homogeneous Markov chain defined on a probability space $(\Omega, \mathcal{F}, \prob)$, with initial distribution $X_0 \sim \gamma_0 \in \cP(\R^d)$. Denote by $P^n$ the $n$-fold iteration of $P$.  
 As usual, we introduce the averaging operator
\begin{equation*}
    (Pf)(x) = \int_{\R^d} f(y) P(x, \dx y)
\end{equation*}
for bounded or nonnegative measurable functions $f$, and similarly the action on measures $\nu$
\begin{equation*}
    (\nu P)(B) = \int_{\R^d} P(y, B) \dx \nu( y), \quad B \in \mathcal{B}(\R^d).
\end{equation*}
We require the following contractivity assumption on the transition kernel.
\begin{assumption} \label{assric}
There are constants $D \geq 1$ and $\kappa \in (0,1)$ such that
\begin{equation*}
    W_1(P^n(x, \cdot), P^n(y, \cdot)) \leq D \kappa ^n |x-y|
\end{equation*}
for all $n \in \N$ and $x,y \in \R^d$.
\end{assumption}
In particular, this means that $P^n(x, \cdot) \in \cP_1(\R^d)$ for all $n \in \N$. In the case $D=1$, this assumption is equivalent to the Markov chain having uniformly positive Ricci curvature in the sense of Ollivier \cite{Ollivier.2009}. In this case, it is enough to require the condition for $n=1$, then the estimate for general $n$ follows by iterated application. Intuitively, the assumption means that if we start two Markov chains at different points $x$ and $y$ then the chains can be coupled in such a way that they approach each other as $n \to \infty$. See \cite{Ollivier.2009} for an introduction to this topic on general metric spaces and several examples of Markov chains which satisfy the assumption. Often one can only ensure the Ricci curvature to be positive by choosing a different (equivalent) metric, which results in the additional factor $D$. This weaker condition is still sufficient for our proofs. \\
The assumption implies a similar relation for arbitrary initial distributions.
\begin{lemma}[\cite{Ollivier.2009}, Proposition 20] \label{w1start}
   If \cref{assric} holds then for any $\mu_0, \nu_0 \in \cP_1(\R^d)$ and $n \in \N$ we have
   \begin{equation*}
       W_1(\mu_0 P^n, \nu_0 P^n) \leq D \kappa ^n W_1(\mu_0, \nu_0).
   \end{equation*}
   
\end{lemma}

Since the space $\cP_1(\R^d)$ is complete, the Banach fixed point theorem implies that there is a unique invariant probability measure $\mu \in \cP_1 ( \R^d )$ for the Markov chain (i.e. $\mu P = \mu$). Moreover, for any initial distribution $\gamma_0 \in \cP_1(\R^d)$  we have $\lim_{n \to \infty} W_1(\gamma_0 P^n, \mu) = 0$ exponentially fast.

\subsection{Rate of convergence in expectation}

Our goal is to find estimates from above for the quantity
\begin{equation*}
\E \left[d_{C^s _ 1 ( \R^d ) } (\mu_n, \mu) \right] = \E \left[ \sup\nolimits_{f \in C^s_1 (\R^d)} |\mu_n(f)-\mu(f)|\right].
\end{equation*}
Hence we need to obtain uniform bounds over $f \in C^s _ 1 ( \R^d )$. To deal with this problem (in the case $s=1$),
Boissard~\cite{Boissard.2011} uses approximations in terms of the covering numbers of the set $\textnormal{Lip}_1(K) = C_1^1 ( K ) $ for a compact subset $K \subset \R^d$. The proofs assume that the transition kernels satisfy a transportation inequality, which is equivalent to an exponential moment condition and therefore a rather strong assumption. 

To obtain uniform bounds in $f$, we will instead follow the approach by Kloeckner~\cite{Kloeckner.2018} and employ an approximation of $f$ by its Fourier series. For $C^s$-functions the Fourier series converges uniformly on compact sets and reasonably fast. Kloeckner uses this to prove that if the Markov chain is supported on a compact set $K \subset \R^d$ 
then there is a constant $C$ depending on $K$, $d$, and $D$ such that for all $n$ large enough one has $\E [d_{C^s_1 ( \R^d ) } (\mu, \mu_n)] \leq C \Gamma_{d , s} (n ' )$, where 
\begin{equation} \label{def:gamma}
\Gamma_{d,s} ( n ' ) =
\begin{cases}
\frac{(\log n')^{d-2 s + \nicefrac{s}{d}}}{(n')^{\nicefrac{s}{d}}},
\quad & d > 2 s \\
\frac{\log n' }{ (n')^{ \nicefrac{1}{2}} } ,
\quad & d = 2 s \\
\frac{(\log n')^{ \nicefrac{d}{2 s + 1}}}{(n')^{\nicefrac{1}{2}}},
\quad & d < 2 s,
\end{cases} 
\end{equation}
and $n'=(1-\kappa)n$ \cite[Theorem 1.1]{Kloeckner.2018}.
In particular, for sufficiently large $s$ the rate of convergence is $n^{ \nicefrac{1}{2}}$ up to logarithmic terms and thus reasonably fast even in high dimensions $d$.
For the case $s=1$, where $d_{C^1_1 } = W_1$, 
the rate is only a power of logarithm slower than the one for i.i.d. random variables with common distribution $\mu$ mentioned before.
To deal with the unboundedess of functions in $C_1^s ( \R^d )$ we will use a truncation argument. For this we need the following moment assumption.

\begin{assumption} \label{assmom}
	There exists a convex, non-decreasing, left-continuous $\Phi \colon [0, \infty ) \to [0, \infty ]$ with $\Phi ( 0 ) = 0$ and 
	$\lim_{x \to \infty} \frac{\Phi ( x ) }{ x } = \infty $
	such that $\sup_{n \in \N_0} \E [ \Phi ( | X_n | ) ] \le 1$.
\end{assumption}

In particular, $\Phi$ is a Young function in the sense of O'Neill~\cite{ONeil.1965}.
Typical examples include $\Phi  ( x ) = x ^q$ for some $q>1$ (corresponding to boundedness of the $q$-th moment)
and $\Phi ( x ) =  e^{\beta x } - 1 $ for some $\beta > 0$.
The special case of a compact support $K \subseteq \R^d$ is also included, by taking $\Phi ( x ) = x$ for $x \le \sup_{y \in K } \| y \|$ and $\Phi ( x ) = \infty $ else. The requirement that the upper bound equals $1$ is arbitrary and can always be obtained by a rescaling of $\Phi$.

Observe that Assumption \ref{assmom} implies that the initial distribution $\gamma_0$ satisfies $\int \Phi ( | x | ) \, \dx \gamma_0 \le 1$. This is the only assumption on $\gamma_0$, no absolute continuity or further regularity is needed. Note also that Assumption \ref{assmom} implies for the invariant measure $\mu$
that
$\int \Phi ( | x | ) \, \dx \mu \le 1$. 

Our main result about the speed of convergence of $\mu_n$ to $\mu$ is the following. We will always write $\lesssim$ for inequalities which hold up to a constant depending on $d$, $s$, $\Phi$, and $D$.
\begin{theorem} \label{thmmarkov}
	Suppose that Assumptions \ref{assric} and \ref{assmom} are satisfied and let $r \in \{ 1, \ldots, s \}$.
	With $n' = (1-\kappa)n$ it holds for all large enough $n$ that
	\begin{equation*}
	\E [d _{C^s _ 1 ( \R^d ) } (\mu, \mu_n)] \lesssim 
	\inf_{R \ge 1} \left[ R^r \Gamma_{d,r } ( n') + \frac{R}{\Phi(R) } \right].
	\end{equation*}
\end{theorem}

Note that $\Gamma_{d,r} (n)$ is defined in \eqref{def:gamma} above.
The speed of convergence in \cref{thmmarkov} is proportional to $ 1 - \kappa $. Hence for $\kappa$ close to $1$ we need more samples of $X_i$ to
obtain a good approximation compared to the independent case.
For the sake of concreteness, let us state two consequences for particular choices of $\Phi$. 
\begin{corollary} Suppose that Assumptions \ref{assric} and \ref{assmom} are satisfied, let $r \in \{ 1, \ldots, s \}$, and write $n' = (1-\kappa)n$.
	\begin{enumerate} [(i)]
		\item If $\Phi(x) = x^q$ for some $q > 1$
		then it holds for all large enough $n$ that
		\begin{equation*}
		\E [d _{C^s _ 1 ( \R^d ) } (\mu, \mu_n)] \lesssim 
		\left( \Gamma_{d,r} ( n') \right)^{\frac{q-1}{q+r-1}}.
		\end{equation*}
		\item If $\Phi(x) = e^{\beta x } - 1$ for some $\beta > 0$
		then
		it holds for all large enough $n$ that
		\begin{equation*}
		\E [d _{C^s _ 1 ( \R^d ) } (\mu, \mu_n)] \lesssim  (\log n')^{ r } \Gamma_{d , r} ( n') .
		\end{equation*}
	\end{enumerate} 
\end{corollary}
The first part can be obtained from Theorem \ref{thmmarkov}
by setting $R = (\Gamma_{d,r} ( n' ) )^{ (1-r-q)^{-1}}$,
which satisfies $R>1$ for $n$ large enough.
For large $q$ this rate comes close to the convergence rate from the compact case.
The second statement follows from Theorem \ref{thmmarkov}
by setting $R = - \beta^{-1} \log \Gamma_{d,r} (n') $, which satisfies $R>1$ for $n$ sufficiently large.
In this case our rate matches the result from the compact case up to a logarithmic power of $n$.
Note that
if $s > \lceil d/2 \rceil $ we can always take $r = \lceil (d+1)/2 \rceil$ to obtain an optimal convergence rate.

To prove \cref{thmmarkov} we will approximate a function $f \in C^s _ 1 ( \R^d ) $ by its Fourier series uniformly on a compact set $K=[-R , R ]^d$. The regularity of $f$ allows us to bound the approximation error
and
the Fourier coefficients independently of $f$. The integrals over the complement of $K$ can by estimated by using the moment condition and the fact that any function in $C^s _ 1 ( \R^d )$ is dominated by $ \Phi ( |x| )$ for large $|x|$ due to the assumptions on $\Phi$. 

In order to verify \cref{assmom} in applications, one can use the following simple criterion.
\begin{prop}
Let $f\colon  \R^d \to [0, \infty ]$ be a measurable function with $\E f(X_0) < \infty$. Suppose that there are constants $C < \infty$ and $\gamma \in (0,1)$ such that $(Pf)(x) \leq \gamma f(x)+C$ for each $x \in \R^d$. Then one has $\sup_{n \in \N_0} \E f(X_n) < \infty$.
\end{prop}
\begin{proof}
We may assume that $\E f(X_0) \leq \frac{C}{1-\gamma}$, otherwise one can replace $C$ with a larger constant. Then we show by induction that $\E f(X_n) \leq \frac{C}{1-\gamma}$ for each $n$. By assumption this holds for $n=0$, and moreover 
\begin{equation*}
    \E f(X_{n+1}) = \E \left[ (P f)(X_n) \right] \leq \gamma \E f(X_n)+C.
\end{equation*}
Hence $\E f(X_n) \leq \frac{C}{1-\gamma}$ implies the same estimate for $n+1$, which completes the proof.
\end{proof}
If the condition $Pf \leq \gamma f+C$ holds for $f ( x ) = \Phi ( | x | )$, then \cref{assmom} of the theorem will be satisfied. 

\subsection{Concentration}
In addition to estimating the expectation of $d _ S (\mu, \mu_n)$ it is also of interest how well the distance concentrates around its expected value. In this section the Markov chain can be supported on an arbitrary Polish metric space $\X$. If the state space is bounded one can use standard bounded difference methods to obtain the following concentration inequality.

\begin{theorem}\label{markcon}
	Suppose that $(X_n)$ is an exponentially contracting Markov chain in the sense of Assumption \ref{assric}, with constants $D = 1$ and $\kappa <1$, taking values in a metric space $\X$ with $\dia (\X) \leq 1$,
	and let $S \subseteq \operatorname{Lip}_1 ( \X)$
	be an arbitrary function class.  Then for all $t \geq 0$ one has
	\begin{equation*}
	\prob \left( d_S (\mu, \mu_n) \geq \E [d_S (\mu, \mu_n)] + t \right) \leq \exp \left( -2(1-\kappa)^2 \cdot nt^2 \right).
	\end{equation*}
\end{theorem}
Note that Theorem \ref{markcon} applies in particular to the $1$-Wasserstein distance with $S = \operatorname{Lip}_1 ( \X )$, the set of $1$-Lipschitz functions.
In the case of i.i.d.\ random variables taking values in a metric space with diameter $ \le 1 $, Weed \& Bach~\cite{Weed.2019} showed that $\prob \left( W_1(\mu, \mu_n) \geq \E [W_1(\mu, \mu_n)] + t \right) \leq \exp \left( -2 nt^2 \right)$.
We obtain the same sub-Gaussian concentration rate, up to the factor $(1- \kappa)^2$. Observe that the rate of concentration does not depend on the dimension or any other specific properties of the state space -- the result holds for an arbitrary bounded metric space. It should be noted that Theorem \ref{markcon} improves on \cite[Theorem 5.4]{Kloeckner.2018} by a factor of $4$ in the exponent. 

In the noncompact case, strong moment assumptions are needed in order to obtain useful concentration inequalities. We say that a measure $\mu \in \cP _1 (\X)$ satisfies a transportation inequality $T_1(C)$ \cite{Bobkov.1999,Bolley.2005} if for all $\nu \in \mathcal{P}(\X)$ with $\nu \ll \mu$ one has
        \begin{equation*}
            W_1(\mu, \nu) \leq \sqrt{2C H(\nu \mid \mu)},
        \end{equation*}
        where $H(\nu \mid \mu)$ is the relative entropy.
Using the Lipschitz properties of the dual metric, this allows us to prove a concentration inequality for $d_S (\mu, \mu_n)$.
\begin{theorem} \label{cor:conctransport}
	Let $(X_n)$ be an exponentially contracting Markov chain on a metric space $\X$ as in Assumption \ref{assric}, with $D=1$ and $\kappa <1$. Suppose that both the initial distribution and the transition kernels $P(x, \cdot)$ satisfy $T_1(C)$ for all $x \in \X$,
	and let $S \subseteq \operatorname{Lip}_1 ( \X)$
	be an arbitrary function class. Then we have for all $t \geq 0$ and $n \in \N$ that
	\begin{equation*}
	\prob( d_S (\mu, \mu_n) \geq \E[ d_S (\mu, \mu_n)]+t) \leq \exp \left( -\tfrac{1}{2C } n t^2 (1-\kappa)^2\right).
	\end{equation*}
\end{theorem}

This is the same sub-Gaussian rate of concentration as in the bounded case.

\section{Proofs}
\subsection{Upper bound for the expectation}
Suppose that the assumptions 1 and 2 are satisfied.
 We first prove some auxiliary statements.
The first lemma we use is adapted from \cite{Kloeckner.2018}, where it is only formulated for Markov chains on compact sets. 
\begin{lemma} \label{hold1}
   Let $\alpha \in (0,1]$ and $f \colon  \R^d \to \mathbb{C}$ be a bounded, $\alpha$-Hölder continuous function. Then for all $m,n \in \N$ one has
   \begin{equation*}
       |\E f(X_n)- \mu (f) | \leq 2D^\alpha \kappa^{\alpha n} \operatorname{Hol}_\alpha(f)
   \end{equation*}
   and
   \begin{equation*}
       |\operatorname{Cov}(f(X_n), f(X_m))| \leq 8D^{\alpha} \kappa^{\alpha |m-n|} \|f\|_\infty \operatorname{Hol}_\alpha(f),
   \end{equation*}
   where $\operatorname{Hol}_\alpha(f)$ denotes the $\alpha$-Hölder constant of $f$.
\end{lemma}
In the case of a compact state space, any Hölder-continuous function is automatically bounded by a value depending on the Hölder constant. But the compactness assumption can be removed if one assumes additionally that $f$ is bounded and combines this with the moment condition.
\begin{proof}
Let $W_\alpha$ be the $1$-Wasserstein distance with respect to the modified metric $|x-y|^\alpha$. 
We first claim that
\begin{equation}\label{estwalp}
    W_\alpha (\mu_0 P^n, \nu_0 P^n) \leq D^\alpha \kappa^{\alpha n} W_\alpha(\mu_0, \nu_0)
\end{equation}
for any $n \in \N$ and probability measures $\mu_0, \nu_0 \in \cP_1$. Indeed, if $\mu_0=\delta_x$ and $\nu_0 = \delta_y$ are Dirac measures then by Jensen's inequality,
\begin{align*}
    W_\alpha(\delta_x P^n, \delta_y P^n) &\leq (W_1 (\delta_x P^n, \delta_y P^n) )^\alpha \leq D^\alpha \kappa^{\alpha n} W_1(\delta_x, \delta_y)^\alpha \\ &=  D^\alpha \kappa^{\alpha n} |x-y|^\alpha =  D^\alpha \kappa^{\alpha n} W_\alpha(\delta_x, \delta_y),
\end{align*}
using the assumption that $P$ is a contraction in $W_1$. Then we can use the same argument as in the proof of \cite[Proposition 20]{Ollivier.2009} to obtain the estimate \eqref{estwalp} in the general case.\\
To prove the first part of the lemma, we use the dual representation for $W_\alpha$, noting that the Hölder constant $\operatorname{Hol}_\alpha(f)$ is the Lipschitz seminorm of $f$ with respect to the metric $|x-y|^\alpha$. The fact that $\mu$ is the stationary measure together with \eqref{estwalp} now implies
\begin{align*}
       |\E f(X_n)- \mu (f) | &= \left| \gamma_0(P^nf)-\mu(f) \right| = \left|\gamma_0(P^nf)-\mu(P^nf) \right| \\
       &\leq \operatorname{Hol}_\alpha(f) W_\alpha(\gamma_0 P^n, \mu P^n) \leq \operatorname{Hol}_\alpha(f) D^\alpha \kappa^{ \alpha n} (W_\alpha(\delta_0, \mu)+W_\alpha(\delta_0, \gamma_0)) \\
       &\leq 2D^\alpha \kappa^{\alpha n} \operatorname{Hol}_\alpha(f).
\end{align*}
Here we used that
\begin{equation*}
    W_\alpha(\delta_0, \mu) = \int |x|^\alpha \ \dx \mu( x) \leq \left( \int |x|^q \ \dx \mu(x) \right)^{\nicefrac{\alpha}{q}} \leq   1
\end{equation*}
by the moment condition and Jensen's inequality, applied to the convex function $u \mapsto u^{q/\alpha}$ on $[0, \infty)$. Analogously, this estimate holds for $W_\alpha(\delta_0, \gamma_0)$, since the initial distribution $\gamma_0$ satisfies the same moment condition. This proves the first assertion.\\ For the second one, note first that if the process starts in an arbitrary point $x \in \R^d$ then
\begin{equation*}
    \left| (P^nf)(x)-\mu(f) \right| \leq \operatorname{Hol}_\alpha(f) D^\alpha \kappa^{\alpha n} W_\alpha(\delta_x, \mu) \leq \operatorname{Hol}_\alpha(f) D^\alpha \kappa^{\alpha n}(|x|^\alpha+1),
\end{equation*}
by the triangle inequality and the same argument as before.
After translating $f$ we may assume that $\mu (f) = 0$. In particular $f$ then takes positive and negative values, so its $L^\infty$-norm increases by a factor of at most $2$. We also assume that $n \geq m$ and write $n=m+t$. Clearly $|f(X_m)| \leq \| f \|_\infty$, and by the first part we get
\begin{equation*}
    |(\E f(X_n))(\E f(X_m))| \leq 2\| f \|_\infty  D^\alpha \kappa^{\alpha n}  \operatorname{Hol}_\alpha(f).
\end{equation*}
For the term $\E[f(X_n)f(X_m)]$, we invoke the Markov property to obtain
\begin{align*}
    |\E [f(X_n)f(X_m)]| &=  |\E \left[ f(X_m) \E[f(X_{m+t}) \mid X_m] \right]| = |\E \left[ f(X_m) (P^tf)(X_m) \right]|\\
    &\leq \operatorname{Hol}_\alpha(f) D^\alpha \kappa^{\alpha t} \left(  \E |f(X_m)| + \E[|f(X_m)| \cdot |X_m|^\alpha] \right) \\
    &\leq 2\| f \|_\infty  D^\alpha \kappa^{\alpha t}  \operatorname{Hol}_\alpha(f),
\end{align*}
where we used again that $\E |X_m|^\alpha \leq 1$ by the moment condition. Combining the two estimates yields the result (the additional factor of $2$ comes from translating $f$).
\end{proof}
Since we want to use a Fourier approximation, we need estimates in terms of the basis functions $e_k(x) = \exp(\pi ik \cdot x)$ for $k \in \Z^d$. As the Lipschitz constant of $e_k$ grows too rapidly as $\| k \|_\infty \to \infty$, we instead apply \cref{hold1} to obtain bounds in terms of the $\alpha$-Hölder constant of $e_k$, for some parameter $\alpha $ to be specified later. The Hölder constant can be bounded as follows.
\begin{lemma}[{\cite[Lemma 4.2]{Kloeckner.2018}}] \label{holek}
   For $\alpha \in (0, 1]$, $k \in \Z^d$ the $\alpha$-Hölder constant of $e_k$ satisfies
   \begin{equation*}
       \operatorname{Hol}_\alpha (e_k) \leq 2^{1-\alpha} \pi^\alpha d^{\alpha/2} \|k\|_\infty^\alpha.
   \end{equation*}
\end{lemma}
\begin{proof}
Since $| \nabla e_k| = \pi |k|$, the function $e_k$ is Lipschitz with constant $\operatorname{Lip}(e_k) \leq \pi |k| \leq \pi \sqrt{d} \|k\|_\infty$. On the other hand, one has $\|e_k\|_\infty = 1$, and thus we obtain for $x \not= y$
\begin{equation*}
    \frac{|e_k(x)-e_k(y)|}{|x-y|^\alpha} \leq \min \left( \frac{2}{|x-y|^\alpha},  \pi \sqrt{d} \cdot \| k \|_\infty  |x-y|^{1-\alpha}\right).
\end{equation*}
If $|x-y| \leq 2(\pi \sqrt{d} \| k \|_\infty )^{-1}$, then the second term does not exceed $2^{1-\alpha} \pi^\alpha d^{\alpha/2} \|k\|_\infty^\alpha$, and otherwise the first term is not larger than this bound. The claim follows.
\end{proof}
These two lemmas enable us to prove the following useful estimate for $|\mu_n(e_k)-\mu(e_k)|$.
\begin{lemma} \label{emunek}
   For all $\alpha \in (0, 1]$, $k \in \Z^d$, $n \geq (1-\kappa^\alpha)^{-1}$ one has
   \begin{equation*}
       \E |\mu_n(e_k)-\mu(e_k)|^2 \lesssim \frac{\|k\|_\infty^{2\alpha}}{n(1-\kappa^\alpha)}.
   \end{equation*}
\end{lemma}
\begin{proof}
Note that
\begin{align*}
       \E |\mu_n(e_k)-\mu(e_k)|^2 &= \E \Bigl[ \left( \tfrac{1}{n} \textstyle{\sum_{j=1}^n} e_k(X_j) - \mu(e_k) \right)^2 \Bigr] \\
       &= \tfrac{1}{n^2} \left(\textstyle{\sum_{j=1}^n}  \bigl(\E [e_k(X_j)]-\mu(e_k)\bigr)  \right)^2 + \tfrac{1}{n^2} \textstyle{\sum_{1 \leq j,l \leq n}} \operatorname{Cov}(e_k(X_j),e_k(X_l))
\end{align*}
Using \cref{hold1} and \cref{holek}, we obtain 
\begin{align*}
\frac{1}{n^2} \sum_{1 \leq j,l \leq n} \left|\operatorname{Cov}(e_k(X_j),e_k(X_l))\right| \leq \frac{16D^\alpha  \pi^\alpha d^{\alpha/2} \|k\|_\infty^\alpha}{n^2} \sum_{1 \leq j,l \leq n} \kappa^{\alpha |j-l|} \lesssim \frac{  \|k\|_\infty^\alpha }{n(1-\kappa^\alpha)}.
\end{align*}
Similarly, for the second term we get
\begin{align*}
    \frac{1}{n^2} \left( \textstyle{\sum_{j=1}^n}  \bigl(\E e_k(X_j)-\mu(e_k)\bigr)  \right)^2 \lesssim \frac{  \|k\|_\infty^{2\alpha}} {n^2}\left( \textstyle{\sum_{j=1}^n} \kappa^{\alpha j} \right)^2 \leq   \|k\|_\infty^{2\alpha} \frac{\kappa^{2\alpha}}{n^2(1-\kappa^\alpha)^2} 
    \leq \frac{  \| k \|_\infty^{2\alpha}}{n(1-\kappa^\alpha)},
\end{align*}
where we used the assumption on $n$ in the last step. The proof is complete.
\end{proof}
Now take an arbitrary function $f \in C_1^s (\R^d)$. To estimate $|\mu_n(f)-\mu(f)|$, we can assume $f(0)=0$. For some radius $R> 1 $ to be determined later, let $K:=[-R,R]^d$. The idea will be to approximate $f$ on $K$ by its Fourier series $g$, and to use the moment condition to obtain bounds for the integral over $\R^d \smallsetminus K = K^c$.
More precisely, we will use the following lemma, which is a direct consequence of the triangle inequality.
\begin{lemma}\label{eztri}
If $g \colon  \R^d \to \mathbb{C}$ is any bounded measurable function, then
\begin{equation*}
    |\mu_n(f)-\mu(f)| \leq 2\|f-g\|_{L^\infty(K)} + |\mu_n(g)-\mu(g)| + \int_{K^c} (|f|+|g|) \ \dx \mu_n +  \int_{K^c} (|f|+|g|) \ \dx \mu.
\end{equation*}
\end{lemma}
Let $\phi \in C^\infty ( \R ^d )$ be such that $\operatorname{supp} ( \phi ) \subseteq [- 2 R , 2 R ] ^d $ and $\phi | _K \equiv 1$. As long as $R > 1$, $\phi$ can be chosen in such a way that all its derivatives are bounded by a constant depending only on $d$.
Then the function $\tilde{f} := f \phi \colon  [-2R,2R]^d \to \mathbb{C}$ has periodic boundary conditions and satisfies $f |_K = \tilde{f}|_K$. Furthermore, $
\| \tilde{f} \|_{C^r_1 } \lesssim\| \tilde{f} \|_{C^s_1 } \lesssim 1$.
By scaling, the function $g \colon  [-1,1]^d \to \R$ defined by $g(x)=\tilde{f}(2Rx)$ has periodic boundary conditions and satisfies $\| g \|_{C^r _ 1 } \lesssim R^r $. 
Let $\mathcal{F}^g(x) = \sum_{k \in \Z ^d} \hat{g_k} e^{\pi ik \cdot x}$
be the Fourier series of $g$, and for given $J \in \N$ let
\begin{equation*}
    \mathcal{F}_J^g(x) = \sum_{k \in \Z ^d, \|k \|_\infty \leq J} \hat{g_k} e^{\pi ik \cdot x}
\end{equation*}
the approximation of order $J$. Then Theorem 4.4 in Schultz~\cite{Schultz.1969} implies for all $J \ge 2$ that
\begin{equation*} 
    \|   \mathcal{F}_J^g - g \|_{L^\infty([-1,1]^d)} \leq C_d R ^r \frac{(\log J)^d}{J ^r},
\end{equation*}
where $C_d \leq C d2^d$ for an absolute constant $C$. Hence if we define $\mathcal{F}_J^f(x) := \mathcal{F}_J^g (x/2R)$, then 
\begin{equation*}
     \| \mathcal{F}_J^f - f \|_{L^\infty([-R,R]^d)} \leq \|   \mathcal{F}_J^f - \tilde{f} \|_{L^\infty([-2R,2R]^d)} \lesssim R ^r \frac{(\log J)^d}{J ^r}.
\end{equation*}
Now we apply \cref{eztri} for $g= \mathcal{F}_J^f$. To estimate the integrals over $K^c$, we will use the moment condition and Orlicz-Hölder's inequality (cf.~\cite{ONeil.1965}).
Let $\Phi^* \colon [0, \infty ) \to [0, \infty )$ denote the Young conjugate of $\Phi$, defined for $x \in [0 , \infty )$ by 
\begin{equation*}
	\Phi ^* ( x ) = \sup\nolimits_{y \ge 0 } ( x y - \Phi ( y ) ).
\end{equation*}
Since $f$ is $1$-Lipschitz with $f(0)=0$, we have $|f(x)| \leq |x|$ for any $ x \in \R^d$ and therefore, by Orlicz-Hölder's inequality and Assumption \ref{assmom},
\begin{align*}
\int_{K^c} |f| \ \dx \mu 
&\leq \int_{K^c} |x| \ \dx \mu
= \int_{\R^d } | x | \mathbf{1}_{K^c } ( x ) \, \dx \mu
\\ &
\le 2 \left( \inf \left\{ \theta > 0 \colon \int_{\R^d } \Phi ( \theta^{-1} | x | ) \, \dx \mu \le 1 \right\} \right)
\left( \inf \left\{ \theta > 0 \colon \int_{\R^d } \Phi^* ( \theta^{-1} \mathbf{1}_{K^c} ) \, \dx \mu \le 1 \right\} \right)
\\ & \lesssim \frac{1}{ (\Phi^* ) ^{-1} ( (\mu ( K^c ) )^{-1 } ) }.
\end{align*}
Here we used that, for $\theta > 0$,
one has $
\int_{\R^d } \Phi^* ( \theta^{-1} \mathbf{1}_{K^c} ) \, \dx \mu
= \mu ( K^c ) \Phi^* ( \theta^{-1})$
and thus the second factor equals $\frac{1}{ (\Phi^* ) ^{-1} ( (\mu ( K^c ) )^{-1 } ) }$ by monotonicity of $\Phi^*$.
On the other hand, since $|x| \geq R$ for $x \in K^c$ the Markov inequality gives $\mu(K^c) \leq (\Phi ( R ) ) ^{-1} $. Thus we obtain
\begin{equation*}
\int_{K^c} |f| \ \dx \mu \lesssim \frac{1}{ (\Phi^*)^{-1} (  \Phi ( R ) ) }.
\end{equation*}
Analogously, for the measure $\mu_n$ we have
\begin{equation*}
	\int_{K^c} |f| \ \dx \mu_n \leq \int_{K^c} |x| \ \dx \mu_n.
\end{equation*}
To estimate the tail integrals for $\mathcal{F}_J^f$, note that by periodicity one has
\begin{align*}
\sup_{x \in K^c} |\mathcal{F}_J^f(x)| = \| \mathcal{F}_J^f \|_{L^\infty([-2R,2R]^d)} \lesssim R^r  \frac{(\log J)^d}{ J ^r } + \|\tilde{f} \|_{L^\infty([-2R,2R]^d)} \lesssim R^r  \frac{(\log J)^d}{J^r } + R ,
\end{align*}
since $\tilde{f}$ is by assumption $1$-Lipschitz. This yields
\begin{equation*}
\int_{K^c} |\mathcal{F}_J^f| \ \dx \mu
\lesssim \mu(K^c) \left( R^r  \frac{(\log J)^d}{J ^r } + R \right) 
\leq    R^r  \frac{(\log J)^d}{J ^r } + R \mu(K^c) 
\leq R^r  \frac{(\log J)^d}{J ^r } +  \frac{R}{\Phi(R)}
.
\end{equation*}
Analogously, we obtain for $\int_{K^c} |\mathcal{F}_J^f| \ \dx \mu_n$ the upper bound
\begin{equation*}
\int_{K^c} |\mathcal{F}_J^f| \ \dx \mu_n \lesssim \left( R ^r  \frac{(\log J)^d}{ J ^r } + R \right) \mu_n(K^c)
\le R^r  \frac{(\log J)^d}{J ^r } + R \mu_n ( K^c ).
\end{equation*}
Hence these terms are of the same order as the corresponding expressions for $|f|$.\\
It remains to estimate the term $|\mu(\mathcal{F}_J^f)-\mu_n(\mathcal{F}_J^f)|$ from above. Writing $e_k^R (x) := e^{\frac{\pi i}{2R} k \cdot x}$,
we obtain
\begin{align*}
    |\mu(\mathcal{F}_J^f)-\mu_n(\mathcal{F}_J^f)| &\leq \sum_{0<\| k \|_\infty \leq J} |\hat{g_k}| |\mu(e_k^R)-\mu_n(e_k^R)| \\
    &\leq \left( \sum_{0<\| k \|_\infty \leq J} \|k\|_\infty^{2 r } |\hat{g_k}|^2 \right)^{\nicefrac{1}{2}} \left(\sum_{0<\| k \|_\infty \leq J} \frac{|\mu(e_k^R)-\mu_n(e_k^R)|^2 }{\|k\|_\infty^{2 r } }\right)^{\nicefrac{1}{2}} \\
    &\leq R^r \left(\sum_{0<\| k \|_\infty \leq J} \frac{|\mu(e_k^R)-\mu_n(e_k^R)|^2 }{\|k\|_\infty^2}\right)^{\nicefrac{1}{2}},
\end{align*}
where we used the Cauchy-Schwarz inequality and the factor $R^r $ comes from the $C^r$-norm of $g$. Moreover, the term for $k=0$ corresponds to a constant function and can therefore be ignored.
Combining all the estimates we obtain the following lemma.
\begin{lemma}
	For all $f \in C_1^s (\R^d)$ and all $R > 1$, $J \ge 2$ it holds that
	\begin{align*}
	|\mu_n(f)-\mu(f)| 
	& \lesssim  R^r  \frac{(\log J)^d}{J^r } + R (\mu_n ( K^c ) + (\Phi (  R ) ) ^{-1} ) +
	\frac{1}{ (\Phi^*)^{-1} ( \Phi ( R ) ) }
	+\int_{K^c} |x| \ \dx \mu_n
	\\
	& \quad + R^r  \left(\sum_{0<\| k \|_\infty \leq J} \frac{|\mu(e_k^R)-\mu_n(e_k^R)|^2 }{\|k\|_\infty^{2 r } }\right)^{\nicefrac{1}{2}}.
	\end{align*}
\end{lemma}
Note that the right-hand side does not depend on $f$ anymore. Hence, the same bound also holds for $\sup_{f \in C_1^s } |\mu_n(f)-\mu(f)| = d_{C_1^s } (\mu_n, \mu)$. In the next step we take the expectation on both sides.
 Note that
\begin{align*}
    \E \left[ \int_{K^c} |x| \ \dx \mu_n\right] = \frac{1}{n} \sum_{i=1}^n \E [|X_i| \mathds{1}_{X_i \in K^c}] \lesssim \frac{1}{ (\Phi^*)^{-1} ( \Phi ( R ) ) },
\end{align*}
since the distribution of each $X_i$ satisfies by assumption the same moment condition as $\mu$. Analogously,
\begin{equation*}
    \E [\mu_n(K^c)] = \frac{1}{n} \sum_{i=1}^n \prob (X_i \in K^c) \lesssim ( \Phi ( R ) ) ^{ - 1 } . 
\end{equation*}
Thus, we get
	\begin{equation}
	\label{eqd1est:weak}
	\begin{split}
		\E [d_{C_1^s ( \R^d ) }(\mu_n, \mu)] 
		&\lesssim R^r  \frac{(\log J)^d}{J ^r } + \frac{R}{\Phi ( R ) } +
		\frac{1}{ (\Phi^*)^{-1} ( \Phi ( R ) ) } \\
		& \quad + R^r \Biggl(\sum_{0< \| k \|_\infty \leq J} \frac{\E |\mu(e_k^R)-\mu_n(e_k^R)|^2 }{\|k\|_\infty^{2 r } }\Biggr)^{\nicefrac{1}{2}}.
	\end{split}
\end{equation}
Next we show that the second term dominates the third one.
\begin{lemma}
	For all $R>0$ we have 
	\begin{equation*}
	 \frac{R}{\Phi (R)} \ge \frac{1}{ (\Phi^*)^{-1} ( \Phi ( R ) )}.
	\end{equation*}
\end{lemma}
\begin{proof}
The claim is equivalent to $(\Phi^*)^{-1} ( \Phi ( R ) ) \ge \frac{\Phi ( R ) }{ R }$.
Thus it suffices to prove that $\Phi ( R ) \ge \Phi ^* \left( \frac{\Phi ( R ) }{R} \right)$.
By definition of $\Phi ^* $, we need to show that $\frac{x \Phi ( R ) }{R} - \Phi ( x ) \le \Phi ( R )$ for every $x \ge 0$. If $x \le R$ this is obviously true. On the other hand, if $x \ge R$ we have $\frac{x \Phi ( R ) }{R} \le \Phi ( x ) $, since $x \mapsto \frac{\Phi ( x ) }{x}$ is increasing. This completes the proof of the lemma.
\end{proof}
Applying this lemma on the right-hand side of \eqref{eqd1est:weak} yields the following.
\begin{corollary}
Under the assumptions of \cref{thmmarkov}, for all $J \in \N$ and $R>0$ one has
	\begin{equation} \label{eqd1est}
\E [d_{C_1^s ( \R^d ) }(\mu_n, \mu)] \lesssim R^r \frac{(\log J)^d}{J ^r } + \frac{R}{\Phi ( R ) }
+ R^r  \Biggl(\sum_{0< \| k \|_\infty \leq J} \frac{\E |\mu(e_k^R)-\mu_n(e_k^R)|^2 }{\|k\|_\infty^{2 r } }\Biggr)^{\nicefrac{1}{2}}.
\end{equation}
\end{corollary}

To prove the theorem it remains to apply the estimate from Lemma \ref{emunek} and then choose $\alpha$, $R$, and $J$ appropriately. The Hölder constant of $e_k^R$ is the constant for $e_k$ divided by $(2R)^\alpha$, but this factor is at most $1$ for $R>1$, so it can be ignored. For any $ \alpha \in (0,1)$ and $n \ge (1 - \kappa^\alpha)^{-1}$ we thus have
\begin{align*}
\Biggl(\sum_{0< \| k \|_\infty \leq J} \frac{\E |\mu(e_k^R)-\mu_n(e_k^R)|^2 }{\|k\|_\infty^{2 r } }\Biggr)^{\nicefrac{1}{2}}  
\lesssim\left( \sum_{0 < \|k \|_\infty \leq J} \frac{ \| k \|_\infty^{2\alpha}}{n(1- \kappa^\alpha) \| k \|_\infty ^{2 r } }\right) ^{\nicefrac{1}{2}} 
\lesssim \left(\sum_{j=1}^J \frac{j^{d+2\alpha - 2 r -1}}{n(1- \kappa ^\alpha)} \right) ^{\nicefrac{1}{2}},
\end{align*}
using that the number of points $k \in \Z^d$ with $\|k\|_\infty = j$ is of order $j^{d-1}$. Now we choose $\alpha = 1/ \log_2 J$, which gives $j^{2\alpha} \leq J^{2\alpha} =4$ for all $1 \leq j \leq J$. Together with $1-\kappa^\alpha \geq \alpha(1-\kappa)$ this implies
\begin{equation*}
\E \left[ \left(\sum_{\| k \|_\infty \leq J} \frac{|\mu(e_k^R)-\mu_n(e_k^R)|^2 }{\|k\|_\infty^{2 r } }\right)^{\nicefrac{1}{2}} \right]
\lesssim \sqrt{\frac{\log J}{n(1-\kappa)}} \left( \sum_{j=1}^J j^{d - 2 r - 1} \right) ^{\nicefrac{1}{2}}
\lesssim J^{\nicefrac{d}{2}-r} \sqrt{\frac{\log J}{n(1-\kappa)}} ,
\end{equation*}
since $\sum_{j=1}^J j^{d-2 r - 1 } =\mathcal{O}(J^{d-2 r})$ if $d \geq 2 r + 1$. Setting $n'=n(1-\kappa)$, we obtain
for all $n \ge (1 - \kappa^\alpha)^{-1}$
that
\begin{equation} \label{ew1est}
\E[ d_{C_1^s ( \R^d ) } (\mu_n, \mu)] \lesssim R ^r \left( \frac{(\log J)^d}{J ^r } 
+ J^{\nicefrac{d}{2}-r } \sqrt{\frac{\log J}{n'}}\right) 
+ \frac{R}{\Phi ( R ) } .
\end{equation}
It remains to choose the parameters $R$ and $J$ in such a way that the right-hand side of this estimate is as small as possible.
First, the optimal choice of $J$ is given by
$J= \lfloor (\log n')^{2-\nicefrac{1}{d}} (n')^{\nicefrac{1}{d}}\rfloor$.
This value can be obtained using the ansatz $J= (n')^\beta$ and ignoring terms of lower order, which leads to the optimal exponent of $n'$ if $\beta = \frac{1}{d}$. Then the estimate can be refined by setting $J = (n')^{1/d} (\log n')^\gamma$, and the optimal power of $\log n'$ is obtained for $\gamma = 2-\frac{1}{d}$.

As $n' \to \infty$, we get $J \to \infty$ and $\alpha \to 0$. 
It is also easy to check that $n ' \geq (1-\kappa^\alpha)^{-1}$ is satisfied if $n$ is large enough (independently of $R$), and hence the above estimates are indeed valid. We thus obtain for all $R \ge 1$
and $n$ sufficiently large
\begin{equation*}
\E [ d_{C_1^s ( \R^d ) } (\mu_n, \mu)] \lesssim R^r \left( \frac{(\log n')^{d-2 r + \nicefrac{r}{d}}}{(n')^{\nicefrac{r}{d}}}\right) + \frac{R}{\Phi(R) } ,
\end{equation*}
establishing the theorem for $d \geq 2 r + 1$.

Analogously, if $d=2 r $ then $\sum_{j=1}^J j^{d-2r - 1} = \mathcal{O}( \log J)$, and thus we obtain the bound
\begin{equation*}
      \E[ d_{C_1^s ( \R^d ) } (\mu_n, \mu)] \lesssim R ^r \left( \frac{(\log J)^d}{J ^r } 
      +\frac{\log J}{ \sqrt{ n'}}\right) 
\end{equation*}
This time the optimal value for $J$ is $J= \lfloor (n' ) ^{ 1 / ( 2 r ) } ( \log n' ) ^{ (d + 1)/ r } \rfloor$, whence the term in brackets is of order $\log n' / ( n ' ) ^{ 1/2 }$.

 Finally, if $d \le 2 r - 1 $ then $\sum _{j=1}^J j^{d-2 r - 1} = \mathcal{O}(1)$, and thus we have
\begin{equation*}
\E[ d_{C_1^s ( \R^d ) } (\mu_n, \mu)] \lesssim R ^r \left( \frac{(\log J)^d}{J ^r } 
+ \sqrt{\frac{\log J}{n'}} \right) 
\end{equation*}
Here the optimal choice for $J$ is $J= \lfloor (n' ) ^{ 1 / ( 2 r ) } ( \log n' ) ^{2 d  / ( 2 r + 1 ) } \rfloor$. and then the term in brackets is of order $ (n' ) ^{ - 1 / 2 } ( \log n' ) ^{ d / (2 r + 1 ) } $. This completes the proof of \cref{thmmarkov}.

\subsection{Concentration inequalities}
To prove \cref{markcon} we will use standard bounded difference arguments. We first formulate a general result concerning concentration of functions of $n$ random variables. Let $Y_1, \ldots, Y_n$ be random variables taking values in $\X$ and let $f \colon \X^n \to \R$ be a bounded function. For $1 \leq i \leq j \leq n$ denote the vector $(Y_i, Y_{i+1}, \ldots, Y_j) =: Y_i^j$. We will also write $Y=Y_1^n=(Y_1, \ldots, Y_n)$. For given elements $y_i \in \X$ define
\begin{equation*}
    \Delta_k(y_1^k) := \E \left[f(Y) \mid Y_1^k = y_1^k \right] - \E \left[ f(Y) \mid Y_1^{k-1} = y_1^{k-1} \right], \quad 1 \leq k \leq n.
\end{equation*}
That is, $\Delta_k$ denotes by how much the expectation of $f(Y)$ changes under the additional information that $Y_k$ takes the value $y_k$.
Next, we set 
\begin{equation*}
    D_k(y_1^{k-1}) = \sup_{x,y \in \X} \left\| \Delta_k(y_1^{k-1}, x) - \Delta_k(y_1^{k-1}, y) \right\|_\infty ,
\end{equation*}
where $(y_1^{k-1},x)$ denotes the vector $(y_1, \ldots, y_{k-1}, x)$. Finally, define
\begin{equation*}
    \mathbf{C} := \sup \left\{ \textstyle{\sum_{k=1}^n} |D_k(y_1^{k-1})|^2 \mid (y_1, \ldots, y_n) \in  \X^n \right\}.
\end{equation*}
Then the following result due to McDiarmid holds.
\begin{lemma} [{\cite[Theorem 3.7]{McDiarmid.1998}}]\label{mcd1}
   Under the above conditions, for any $t \geq 0$ one has
    \begin{equation*}
        \prob( f(Y) \geq\E f(Y)+ t) \leq \exp \left( - \tfrac{2t^2}{\mathbf{C}} \right).
    \end{equation*}
\end{lemma}

In order to apply this to the dual distance $d_S (\mu, \mu_n)$, 
we define on $\X ^n$
the distance
\begin{equation} \label{eq:lip:d1n}
d_n^{ ( 1 ) } ( ( x_1, \ldots, x_n ) , (x_1', \ldots, x_n' ) ) = \sum_{k=1}^n d ( x_k, x_k' ).
\end{equation}
Now we use that the function $f\colon  \X ^n \to \R$ defined by $ f(x_1, \ldots, x_n) :=d_S (\mu, \mu_n)$ is Lipschitz with constant $\frac{1}{n}$. That is, for all $ x_1, \ldots, x_n, x_1', \ldots, x_n' \in \X$ we have
\begin{equation*} 
    |f(x_1,  \ldots, x_n)-f(x_1', \ldots, x_n')| \leq \frac{1}{n} \sum_{k=1}^n d(x_k, x_k')
     = \frac{1}{n} d_n^{ ( 1 ) } ( ( x_1, \ldots, x_n ) , (x_1', \ldots, x_n' )).
\end{equation*}
This follows directly from the triangle inequality and the fact that $S \subseteq \operatorname{Lip}_1 ( \X )$. 

To apply \cref{mcd1} to the function $f$, we first need a general estimate for Lipschitz functions of the Markov chain.
\begin{lemma} \label{markovlip}
   Let $n \in \N$ and $f \colon \X^{n+1} \to \R$ be a function which is Lipschitz with respect to the metric $d_{n+1}^{(1)}$ with constant $L$. Then the function $F \colon  \X \to \R$ defined by
   \begin{equation*}
       F(x) = \E^x f(X_0, X_1, \ldots, X_n)
   \end{equation*}
   is Lipschitz with constant $L \sum_{j=0}^n \kappa^j \leq \frac{L}{1-\kappa}$.
\end{lemma}
Here $\E^x$ denotes the expectation for the Markov chain starting in $X_0=x$, and $d^{(1)}_{n+1}$ is given by \eqref{eq:lip:d1n}.
\begin{proof}
We show the statement by induction on $n$. If $n=1$ we have for $x,y \in \X$ that
\begin{align*}
    |F(x)-F(y)| &= \left|\E^x f(x,X_1)- \E^y f(y, X_1) \right| \\ &\leq |\E^x[ f(x,X_1)- f(y,X_1)]|+|\E^x f(y,X_1)-\E^y f(y,X_1)| \\
    &\leq Ld(x,y)+ \left| \int f(y,z) P(x, \dx z) - \int f(y,z) P(y, \dx z) \right| \\ &\leq L(1+\kappa)d(x,y),
\end{align*}
where we used that $W_1(P(x, \cdot), P(y, \cdot)) \leq \kappa d(x,y)$ and the Kantorovich duality for $W_1$.

Now suppose that the statement holds for some $n$, and let $F \colon  \X^{n+2} \to \R$ be Lipschitz with constant $L$. For each $x \in \X$ define $F_x'(z) := \E^z f(x, X_0, \ldots, X_n)$, then by the induction hypothesis $z \mapsto F_x'(z)$ is Lipschitz with constant $L \sum_{j=0}^n \kappa ^j$. The Markov property now implies that for $x,y \in \X$
\begin{align*}
    |F(x)-F(y)| \leq& \ |\E^x [f(x, X_1, \ldots, X_{n+1})-f(y, X_1, \ldots, X_{n+1})]| \\&+ |\E^x f(y, X_1, \ldots, X_{n+1})-\E^y f(y, X_1, \ldots, X_{n+1})| \\ \leq & \ Ld(x,y)+ \left| \int F_y'(z) P(x, \dx z) - \int F_y'(z) P(y, \dx z) \right| \\
    \leq & \ \left(L+ \kappa L \textstyle\sum_{j=0}^n \kappa ^j \right) d(x,y) = L \textstyle\sum_{j=0}^{n+1} \kappa^j d(x,y),
\end{align*}
where we again applied the duality. This completes the proof.
\end{proof}
\begin{remark}
Similar ideas have been used in \cite{Dedecker.2015} to estimate the concentration for separately Lipschitz functions of Markov chains. 
Here we really need to assume the contractivity with $D=1$, since otherwise the inductive proof would not be possible.
\end{remark}
\begin{proof}[Proof of \cref{markcon}]
For $1 \leq i \leq j$ we set $X_i^j := (X_i, X_{i+1}, \ldots, X_j)$. As in \cref{mcd1}, for given $x_1, \ldots, x_k \in \X$ we define
\begin{equation*}
    \Delta_k(x_1^k) =\E [f(X_1^n) \mid X_1 = x_1, \ldots, X_k=x_k] - \E[f(X_1^n) \mid X_1=x_1, \ldots, X_{k-1}=x_{k-1}].
\end{equation*}
Now for all $x,y \in \X$ the Markov property implies
\begin{align*}
\Delta_k(x_1^{k-1}, x) - \Delta_k(x_1^{k-1},y)  &= \E [f(x_1^{k-1}, X_k^n) \mid X_k=x] - \E [f(x_1^{k-1}, X_k^n) \mid X_k=y] \\
    &= \E^x f(x_1, \ldots, x_{k-1}, x, X_1, \ldots, X_{n-k}) -  \E^y f(x_1, \ldots, x_{k-1}, y, X_1, \ldots).
\end{align*}
Next, we apply \cref{markovlip}, which leads to
\begin{align*}
    \left| \Delta_k(x_1^{k-1}, x) - \Delta_k(x_1^{k-1},y)\right| &\leq \frac{1}{n} \sum_{j=0}^{n-k} \kappa ^j d(x,y) \leq \frac{1}{n(1-\kappa)} d(x,y).
\end{align*}
Since the space $\X$ is bounded by assumption, we have $d(x,y) \leq 1$. Hence \cref{mcd1} can be applied with $\mathbf{C}=\frac{1}{n(1-\kappa)^2}$ to obtain the result.
\end{proof}

It remains to prove the concentration inequality under transportation assumptions. For this we shall use the well-known fact that a transportation inequality implies sub-Gaussian measure concentration \cite{Bobkov.1999}.
\begin{prop} \label{conct1}
If $\mu$ satisfies $T_1(C)$ then it holds for all $f \in \operatorname{Lip}(\X)$ and $t \geq 0$ that
\begin{equation*}
    \mu(f \geq \mu(f) + t) \leq \exp \left( - \tfrac{t^2}{2C \| f \|_{\operatorname{Lip}}^2} \right).
\end{equation*}
\end{prop}

\begin{proof} [Proof of \cref{cor:conctransport}]
We know that the function
\begin{equation*}
    f \colon  \X^n \to \R, \quad (x_1, \ldots, x_n) \mapsto d_S (\mu, \mu_n), \quad \mu_n = \frac{1}{n} \sum_{k=1}^n \delta_{x_k}
\end{equation*}
is $\frac{1}{n}$-Lipschitz with respect to the metric $d_n^{(1)}$ on $\X^n$. Furthermore, \cite[Theorem 1.1]{Blower.2006} implies that the distribution of $(X_1, X_2, \ldots, X_n)$ satisfies $T_1(C_n)$ with respect to the metric $d_n^{(1)}$ on $\X^n$, where
\begin{equation*}
    C_n = C \sum_{m=1}^n \left( \sum_{k=0}^{m-1} \kappa ^k \right)^2 \leq  C n \left( \sum_{k=0}^\infty \kappa ^k \right)^2 = \frac{C n}{(1-\kappa)^2}.
\end{equation*} Combining this with \cref{conct1} completes the proof.
\end{proof}

\subsection*{Acknowledgments}
 The author is supported by the Deutsche
 Forschungsgemeinschaft (DFG, German Research Foundation) under Germany’s Excellence Strategy EXC 2044–390685587, Mathematics Münster: Dynamics–Geometry–Structure.
This work is based on the author's master's thesis, which was conducted at Bonn University. The author is grateful to his advisor Andreas Eberle for his help and support.

\end{document}